\documentclass[twoside,11pt]{article}

\usepackage{graphicx, subfigure}
\usepackage[top=1in,bottom=1in,left=1in,right=1in]{geometry}
\usepackage[sort&compress]{natbib} 
\usepackage{amsfonts, amsmath, amssymb, amsthm, bbm}
\usepackage[normalem]{ulem}
\usepackage{comment}
\usepackage{eufrak}

\usepackage{color}
\definecolor{darkred}{RGB}{100,0,0}
\definecolor{darkgreen}{RGB}{0,100,0}
\definecolor{darkblue}{RGB}{0,0,150}

\usepackage{hyperref}
\hypersetup{colorlinks=true, linkcolor=darkred, citecolor=darkgreen, urlcolor=darkblue}
\usepackage{url}

\newtheorem{thm}{Theorem}
\newtheorem{prp}{Proposition}
\newtheorem{lem}{Lemma}
\newtheorem{cor}{Corollary}
\theoremstyle{remark}
\newtheorem{rem}{Remark}

\def\beq{\begin{equation}}
\def\eeq{\end{equation}}
\def\beqn{\begin{eqnarray*}}
\def\eeqn{\end{eqnarray*}}
\def\bitem{\begin{itemize}}
\def\eitem{\end{itemize}}
\def\benum{\begin{enumerate}}
\def\eenum{\end{enumerate}}
\def\bmult{\begin{multline*}}
\def\emult{\end{multline*}}
\def\bcenter{\begin{center}}
\def\ecenter{\end{center}}

\newcommand{\thmref}[1]{Theorem~\ref{thm:#1}}
\newcommand{\prpref}[1]{Proposition~\ref{prp:#1}}
\newcommand{\corref}[1]{Corollary~\ref{cor:#1}}
\newcommand{\lemref}[1]{Lemma~\ref{lem:#1}}
\newcommand{\secref}[1]{Section~\ref{sec:#1}}

\DeclareMathOperator{\tr}{tr}

\def\cC{\mathcal{C}}

\def\cE{\mathcal{E}}

\def\cG{\mathcal{G}}
\def\cH{\mathcal{H}}

\def\cN{\mathcal{N}}

\def\cV{\mathcal{V}}

\def\bA{\mathbf{A}}
\def\bB{\mathbf{B}}

\def\bI{\mathbf{I}}

\newcommand\bGamma{{\boldsymbol\Gamma}}

\def\bbI{\mathbb{I}}

\def\bbP{\mathbb{P}}

\def\bbR{\mathbb{R}}

\def\bbZ{\mathbb{Z}}

\newcommand{\E}{\operatorname{\mathbb{E}}}
\renewcommand{\P}{\operatorname{\mathbb{P}}}
\newcommand{\Var}{\operatorname{Var}}
\newcommand{\Cov}{\operatorname{Cov}}

\def\Bin{\text{Bin}}

\newcommand{\PROB}{\bbP}

\newcommand{\wt}{\widetilde}

\def\tr{\operatorname{Tr}}

\newcommand{\1}{{\rm 1}\kern-0.24em{\rm I}}

\newcommand{\IND}[1]{\bbI\{ #1 \}}
\def\X{X}

\begin{document}

\title{Detecting a Path of Correlations in a Network}
\author{
Ery Arias-Castro\footnote{Department of Mathematics, University of California, San Diego, United States} 
\and
G\'abor Lugosi\footnote{ICREA and Department of Economics and Business, Pompeu Fabra University, Barcelona, Spain, %gabor.lugosi@upf.edu
} \and
Nicolas Verzelen\footnote{INRA, UMR 729 MISTEA, F-34060 Montpellier, France}
} 
\date{}
\maketitle

\begin{abstract}
We consider the problem of detecting an anomaly in the form of a path of correlations hidden in white noise.  We provide a minimax lower bound and a test that, under mild assumptions, is able to achieve the lower bound up to a multiplicative constant.  
\end{abstract}

\section{Introduction} \label{sec:intro}

Anomaly detection arises in many applications, including surveillance, the detection of suspicious objects from satellite images or sensor networks, as well as in medical imaging (e.g., tumor detection).
While in some applications the object can be assumed to present a larger signal amplitude (e.g., pixel level in images), in other settings it manifests instead as correlations. For example, the object to be detected in an image has different texture but same average pixel amplitude; or in the case of the evolution of the price of a stock, an event could trigger more volatility instead of a change in the value of the stock.
We call the problem of detecting the presence of a subset of observations with different mean from the
rest the {\em detection-of-means} problem. 
We call the problem of detecting the presence of a subset of unusually correlated observations the {\em detection-of-correlations} problem.

The detection-of-means problem has been extensively studied in the literature, both applied and theoretical. Papers that develop theory include \citep{maze,combin,cluster,MGD,MR2604703,morel}.
The detection-of-correlations problem has drawn less attention, and while the applied literature is sizable, few papers develop theory beyond the one-dimensional case of change-point detection in times series.
In a few recent papers, we developed elements of the first minimax theory, see \citep{correlation-detect,multidim,arias2015detecting}.
In the present paper we focus on detecting a path of correlations in a general graph.  This setting could model an attack in a computer network \citep{mukherjee1994network,zhang2000intrusion}.  
%We develop some minimax theory for this problem.
The corresponding detection-of-means setting was considered in \citep{maze}.

The remainder is organized as follows.  In \secref{general}, after formalizing the problem, we derive a lower bound when the correlation parameter is known and then propose a testing procedure which achieves this lower bound (up to a multiplicative constant depending on some graph characteristics) without knowledge of the correlation parameter.
In \secref{lattice} we specialize our general results to the case of detecting a path of correlations in an
integer lattice.
%We briefly discuss our result and some open questions in \secref{discussion}.
The proofs are gathered in \secref{proofs}.

\section{Setting and general results} \label{sec:general}

\subsection{Formulation of the problem}

We are given a graph $\cG = (\cV, \cE)$, with $\cV$ denoting the set of nodes (or vertices) and $\cE \subset \cV \times \cV$ denoting the set of edges.
When finite, let $n = |\cV|$ denote the number of nodes.
We are also given a class of open, self-avoiding paths of $\cG$, denoted $\cC$.
Recall that an open self-avoiding path in $\cG$ is a sequence of nodes $S = (s_1, \dots, s_k) \in \cV^k$ such that $(s_j, s_{j+1}) \in \cE$ for all $j$ and $s_j \ne s_{j'}$ for all $j \ne j'$.
We observe a vector of random variables indexed by $\cV$, denoted $X = (X_i)_{i \in \cV}$, assumed to be standard normal.
Under the null hypothesis all components of $X$ are independent.
Under the alternative hypothesis, one of the paths $S \in \cC$ is ``anomalous'',
in which case $(X_i)_{i \in S}$ is an autoregressive model of order $1$ with correlation coefficient $\psi \in (-1,1)$, while the other components of $X$ are still independent and $(X_i)_{i \in S}$ and $(X_i)_{i \notin S}$
are also independent.
This means that, if $S = (s_1, \dots, s_k)$, then $X_{s_{j+1}}-\psi X_{s_j}, j = 1, \dots, k-1,$ are independent centered normal random variables with variance $1-\psi^2$. 
We denote the distribution of $\X$ under $\cH_0$ by $\PROB_0$.
We denote the distribution of $\X$ under $\cH_1$ by $\PROB_{S,\psi}$ when $S\in \cC$ is the anomalous set and $\psi$ is the autocorrelation coefficient.  

A \emph{test} is a measurable function $f: \bbR^\cV \to \{0,1\}$. When $f(\X)=0$, 
the test accepts the null hypothesis and it rejects it otherwise.
The probability of \emph{type I} error of a test $f$ is $\PROB_0\{f(\X)=1\}$.
Under the alternative $(S,\psi) \in \cC \times (-1,1)$, the probability of \emph{type II} error is $\PROB_{S,\psi}\{f(\X)=0\}$.
In this paper we evaluate tests based on their \emph{worst-case risks}.
When the correlation coefficient $\psi$ is known, the risk of a test $f$ corresponding to the class $\cC$ is defined as 
\[
R_{\cC,\psi}(f) = \PROB_0\{f(\X)=1\} +  \max_{S \in \cC} \, \PROB_{S,\psi}\{f(\X)=0\}~.
\]
In this case, the minimax risk is defined as 
\[
R^*_{\cC,\psi} = \inf_f R_{\cC,\psi}(f)~,
\]
where the infimum is over all tests $f$.
When $\psi$ is only known to belong to an interval $\mathfrak{I} \subset (-1,1)$, it is more meaningful to define the risk of a test $f$ as
\[
R_{\cC, \mathfrak{I}}(f) = \PROB_0\{f(\X)=1\} +  \max_{\psi \in \mathfrak{I}} \max_{S \in \cC} \, \PROB_{S,\psi}\{f(\X)=0\}~.
\]
The corresponding minimax risk is defined as
\[
R_{\cC, \mathfrak{I}}^* = \inf_f R_{\cC, \mathfrak{I}}(f)~.
\]
When $\psi$ is known (resp.~unknown), we say that a test $f$ \emph{asymptotically separates the two hypotheses} if $R_{\cC,\psi}(f) \to 0$ (resp.~$R_{\cC, \mathfrak{I}}(f) \to 0$), and we say that the hypotheses \emph{merge asymptotically} if $R_{\cC,\psi}^* \to 1$ (resp.~$R_{\cC, \mathfrak{I}}^* \to 1$), as $n = |\cV| \to \infty$.  We note that, as long as $\psi \in \mathfrak{I}$, $R_{\cC,\psi}^* \le R_{\cC, \mathfrak{I}}^*$ and that $R_{\cC, \mathfrak{I}}^* \le 1$, since the test $f \equiv 1$, that always rejects $\cH_0$, has risk equal to $1$.

In this paper, we characterize the minimax testing risk in the setting where $\psi$ is known ($R^*_{\cC,\psi}$) and in the setting when it is unknown ($R_{\cC, \mathfrak{I}}^*$).  
That is, we give conditions on $\cC$ and $\mathfrak{I}$ under which the hypotheses merge asymptotically so that the detection problem is nearly impossible. 
We then exhibit a (nonstandard) test that asymptotically separates the hypotheses under essentially the same conditions.

\subsection{A general lower bound} \label{sec:lower}

The main difference between the case of anomalous paths treated here and the case of anomalous blobs studied in \cite{arias2015detecting} is that a lower bound for the latter can be developed based on a subclass of disjoint subsets.  Here, however, a reduction to a subclass of disjoint paths is typically too severe.  We thus develop a new lower bound tailored to the present situation, which, in particular, allows for possible anomalous subsets to intersect.
For any prior distribution $\nu$ on $\cC$, the minimax risk is at least as large as the $\nu$-average risk, $R^*_{\cC,\psi} \ge \bar{R}^*_{\nu,\psi}$, where
\beq \label{risk-nu-pi}
\bar{R}_{\nu,\psi}(f) = \PROB_0\{f(\X)=1\} + \sum_{S\in \cC} \nu(S)  \PROB_{S,\psi}\{f(\X)=0\} \quad \text{and} \quad \bar{R}^*_{\nu,\psi} = \inf_f \bar{R}_{\nu,\pi}(f)~.
\eeq
The following result provides a lower bound on the latter. 

\begin{thm} \label{thm:graph}
Let $\cC$ be a class of open self-avoiding paths of $\cG$ and let $\nu$ denote some prior over $\cC$.  Assume that $|\psi| < 1/9$.
Then 
\[
\bar{R}_{\nu, \psi}^* \ge 1 - \frac12 \sqrt{\E_{\nu \otimes \nu}\left[\exp\big(\lambda(\psi) |S \cap T|\big)\right] - 1}~, 
\]
where
\[
\lambda(\psi) := \frac14 \left[ \left(\frac{1-|\psi|}{1-9|\psi|}\right)^{1/2} - \frac{1+|\psi|}{1-|\psi|}\right]
\]
and the expectation is with respect to $S,T$ drawn i.i.d.~from $\nu$. 
\end{thm}

The function $\lambda$ is even and increasing on $(0, 1/9)$, and $\lambda(1/10) = 4/9$.  Hence, the bound implies that  
\[
\bar{R}_{\nu, \psi}^* \ge 1 - \frac12 \sqrt{\E_{\nu \otimes \nu}\left[\exp\big(\tfrac49 |S \cap T|\big)\right] - 1}~, \quad \text{when } |\psi| \le 1/10.
\]
%Moreover, $\lambda(\psi) \sim 6 \psi^2$ as $\psi \to 0$.

\begin{rem}
The condition on $\psi$ is likely to be an artifact of our proof technique.  As is standard in the literature, the proof relies on bounding the variance of the likelihood ratio resulting from averaging the alternatives according to $\nu$.  A more refined approach such as that of bounding the second moment of the likelihood ratio after truncation---a well-known technique in the detection-of-means setting first proposed by Yuri Ingster---might lead to a sharpening of this result, but the computations for the present case are daunting.
\end{rem}

\subsection{A general upper bound} \label{sec:upper}

A natural approach in related testing problems is the generalized likelihood ratio test.
When $\psi$ is known, this test is based on rejecting the null hypothesis for large values of 
\[\max_{S \in \cC} X_S^\top (\bI_S - \bGamma_S^{-1}(\psi)) X_S~,\]
where $\bGamma_S(\psi)$ denotes the covariance matrix of an autoregressive model of order 1
indexed by $S$ and with parameter $\psi$ and $X_S=\sum_{i\in S} X_i$.
Establishing a useful performance bound for this test appears surprisingly challenging due to our lack of understanding of concentration properties of the test statistic under the null hypothesis.  In particular, our effort to combine the union bound with a standard  concentration bound for Gaussian quadratic forms (i.e., Gaussian chaoses of order $2$) were inconclusive.
The situation is even more complicated when $\psi$ is unknown.

However, we were able to craft and analyze an ad-hoc test based on pairwise comparisons of consecutive values along a path.  
For simplicity assume that all paths in $\cC$ are of same length $k$.  (When this is not the case, typically the test needs to be repeated for all possible lengths and the resulting multiple testing situation is resolved by applying Bonferroni's method.)
Fix a threshold $t > 0$.  For $S = (s_1, \dots, s_k) \in \cC$,  define
\[
V_{t,S} = \sum_{j = 1}^{k-1} V_{t,S}(j)~, \quad V_{t,S}(j) = \IND{|X_{s_{j+1}} - X_{s_j}| \le \sqrt{2} t}~,
\]
and consider the statistic and corresponding test 
\beq\label{def_test_vt}
V_t^* = \max_{S \in \cC} V_{t,S}~, \quad f_t = \IND{V_t^* > k/2 }~.
\eeq

%\nv{Perhaps, we could remove the following remark}
%\medskip \noindent {\em Remark.}
%Computing $V_t^*$ above is difficult in general.  A special case where this is feasible in polynomial time is when the graph can be represented as a directed acyclic graph and the class $\cC$ is that of all paths with exactly $k$ nodes in the graph.  In that case, computing $V_t^*$ can be done efficiently by dynamic programming.

For $t \ge 0$, define $p_t = 2 \mathsf{N}(t) - 1$, where $\mathsf{N}$ denotes the standard normal distribution function.  Also, define the function $h(x) = x -\log(x)-1$.

\begin{prp} \label{prp:V}
There is a sequence of reals $(r_k)$, with $r_k \to 0$ as $k \to \infty$, such that the following is true.
Consider any setting where $\cC$ is a class of (self-avoiding) paths of length $k$.
Let $t = t(k) > 0$ be largest such that $h(2p_t)\geq \frac8k \log(|\cC|) \vee 1$.
If $\psi \ge 1 - (t/\mathsf{N}^{-1}(4/5))^2$, the test $f_t$ defined in \eqref{def_test_vt} satisfies $R_{\cC, \psi}(f_t) \le r_k$. 
\end{prp}

\begin{rem}
\label{rem:negative-psi}
The proposition only applies to positive $\psi$ and in fact the test \eqref{def_test_vt} is only useful in that case.  
To handle the case where $\psi$ is negative, we use instead the variant where $V_{t,S}(j)$ is replaced with $\IND{|X_{s_{j+1}} + X_{s_j}| \le \sqrt{2} t}$.
The resulting test achieves a similar performance.
In practice, if the sign of $\psi$ is a priori unknown, one can simply combine these two tests using a Bonferroni correction.  In the rest of the paper we only consider $\psi>0$.
\end{rem}

\begin{rem}
Computing the test statistic $V_t^*$ of \eqref{def_test_vt} is difficult, even when the starting point is given.  Indeed, this problem is known as the {\em prize collecting salesman problem} or {\em bank robber problem} or {\em reward-budget problem}.  Even in the case where the underlying graph is the integer lattice there are no known polynomial-time algorithms that solve it, although polynomial approximations do exist  \cite{DasGuptaHespanhaRiehlSontag06}.
An alternative is the test based on the length of the longest path of significant adjacent correlations. 
This is inspired from some proposals in the detection-of-means setting \cite{arias2013cluster,arias2006adaptive}.
Quick calculations suggest that the test achieves a comparable theoretical performance.
\end{rem}

\section{Special case: the lattice} \label{sec:lattice}

Consider the integer lattice 
\beq \label{lattice}
\cV = \{1, \dots, m\}^d~,
% \subset \cV_\infty = \bbZ^d~,
\eeq
in dimension $d \ge 3$.  Note that $n = m^d$ in this case.
The story is a little different when $d = 2$, and we refer the reader to the treatment in \cite{maze} carried out in the detection-of-means setting, as we expect a similar phenomenon to hold in the present context.
For simplicity, to guarantee that all nodes play a symmetric role, we take the lattice to be a torus.

\subsection{Known starting point}
Let $\cC$ be the class of all self-avoiding paths with $k$ nodes in $\cV$ starting at some given $v_0 \in \cV$.  In that case, when $d \ge 3$, there is a constant $C > 0$ such that, when $|\psi| \le C$, the risk is at least $1/2$.    To see this, let $\nu$ be a prior on $\cC$ that has exponential intersection tails, which means that there exist some constants $\eta \in (0,1)$ and $C_0 > 0$ such that
\beq\label{EIT}
\P_{\nu \otimes \nu}(|S \cap T| \ge \ell) \le C_0 \eta^\ell, \quad \forall \ell \ge 1~,
\eeq
where $S,T$ are i.i.d.~from $\nu$.  
This concept was introduced in \cite[Theorem 1.3]{MR1634419}, where it is shown that such a prior exists on infinite paths ($k = \infty$) in the infinite $d$-dimensional integer lattice ($m=\infty$) when $d = 3$.  In fact, it is constructed with support on oriented paths taking steps in $\{(1,0,0), (0,1,0), (0,0,1)\}$.
Obviously, in the finite case, it suffices to restrict such a prior on the first $k-1$ steps and property \eqref{EIT} still holds when $k \le m$, which we assume henceforth.\footnote{Clearly, an upper bound on $k$ is needed for there are no self-avoiding paths when $k > n$.}  
Note that we can use the same prior when $d \ge 3$, by embedding the 3-dimensional integer lattice into the $d$-dimensional integer lattice.
Because of this, we can take $\eta$ to be a numeric constant  not depending on the dimension $d$.

Thus consider a prior $\nu$ satisfying \eqref{EIT}.
Then, for any $a > 0$ small enough that $e^a \eta < 1$, 
\begin{align*}
\E_{\nu \otimes \nu}[\exp(a |S \cap T|)] 
&=  \sum_{\ell \ge 1} e^{a \ell} \P_{\nu \otimes \nu}(|S \cap T| = \ell) \\
&=  e^a + \sum_{\ell \ge 2} (e^{a \ell} - e^{a (\ell-1)}) \P_{\nu \otimes \nu}(|S \cap T| \ge \ell) \\
&\le  \Xi(a) := e^a + C_0 \frac{(e^a-1)e^a \eta^2}{1 - e^a \eta}~. 
\end{align*}
We use this upper bound in \thmref{graph} with $a = \lambda(|\psi|)$ and get that $\bar{R}_{\nu, \psi}^* \ge 1/2$ when $|\psi|$ is sufficiently small that $\eta \exp(\lambda(|\psi|)) < 1$ and $\Xi(\lambda(|\psi|)) \le 2$.  This is possible since $\lambda(|\psi|) \to 0$ as $|\psi| \to 0$ and $\Xi(a) \to 1$ as $a \to 0$.
%We conclude that, even if we know the starting point of the possible anomalous paths, the risk is bounded away from zero as soon as $\psi$ is small enough. 

Conversely, there exists a positive constant depending only on the dimension $d$ such that, if $\psi$ is larger than that constant, then the test defined in \eqref{def_test_vt} asymptotically separates the hypotheses.  This comes from a simple application of \prpref{V} together with the fact that, in the $d$-dimensional integer lattice, there are at most $(2d)^{k-1}$ paths of length $k$ starting at a given node.
Following Remark~\ref{rem:negative-psi}, we can handle the case when $|\psi|$ is large enough in a similar way.
From this discussion, we arrive at the following.

\begin{cor} \label{cor:known}
Consider the $d$-dimensional integer lattice \eqref{lattice} with $d\ge 3$ seen as a torus.  Let $\cC$ denote the class of self-avoiding paths of length $k \le m$ starting at a known location.  Assume the autocorrelation coefficient $\psi$ is fixed.  
There exist constants $0 < C_1\le C_2 < 1$ depending only on $d$ such that, when $|\psi| < C_1$, $\liminf_{k \to \infty} R^*_{\cC,\psi} > 0$, while when $|\psi| > C_2$, $\lim_{k \to \infty} R^*_{\cC,\psi} = 0$.
\end{cor}

\begin{rem}
As long as $k \le m$, the size of the lattice does not matter since the starting location is known.  Therefore, an asymptotic analysis is necessarily in terms of large $k$.
\end{rem}

\begin{rem} \label{rem:conj1}
We conjecture that the constants in \corref{known} are identical, meaning, that $C_1 = C_2$.  Our arguments show that this is so if we replace $\liminf$ with $\limsup$ in the statement, for in that case one can take 
\beq \label{C-ddag}
C_1 = C_2 = C_\ddag := \inf\big\{C > 0 : \lim_{k \to \infty} R^*_{\cC,\psi} = 0 \text{ when } |\psi| > C\big\}~,
\eeq
and the corollary implies that $C_\ddag \in (0,1)$.
\end{rem}

\subsection{Unknown starting point}
 Let $\cC$ be the class of all self-avoiding paths with $k$ nodes in $\cV$.  Assume that $m=n^{1/d}$ is a multiple of $2k$ for simplicity.
To define the prior, we partition the lattice into hypercubes of side length $2 k - 1$, indexed by $J$, and let $v_j$ denote the center of the hypercube $j \in J$.  The number of such hypercubes satisfies $|J| \sim (m/2k)^d = n/(2k)^d$.  
Still in dimension $d \ge 3$, let $\nu_j$ be a prior on self-avoiding paths starting at $v_j$
satisfying \eqref{EIT}
and let $\nu$ be the even mixture of all these priors.  Noting that paths starting from different origin nodes cannot intersect, for any $a > 0$
small enough that $e^a \eta < 1$, 
\[
\E_{\nu \otimes \nu} [\exp(a |S \cap T|)] 
= 1 - \frac1{|J|} + \frac1{|J|} \E_{\nu_j \otimes \nu_j} [\exp(a |S \cap T|)] 
\le 1 + \frac{\Xi(a)}{|J|}~,
\]
where $j \in J$ is arbitrary. 
We use this upper bound in \thmref{graph} with $a = \lambda(|\psi|)$ and get that $\bar{R}_{\nu, \psi}^* \ge 1 - \frac12 \sqrt{2/|J|}$ when $|\psi|$ is so small that $\eta \exp(\lambda(|\psi|)) < 1$ and $\Xi(\lambda(|\psi|)) \le 2$.  
Noting that $|J| \to \infty$ when $m \gg k$ (i.e., the size of the grid dominates the path length), we see that $\bar{R}_{\nu, \psi}^* \to 1$.  

Conversely, assume that $k/\log n \ge C_0$ for some positive constant $C_0$. 
Then there exists another positive constant depending only on the dimension $d$ and $C_0$ such that, if $\psi$ is larger than that constant, the test defined in \eqref{def_test_vt} asymptotically separates the hypotheses.  This comes from \prpref{V} and the fact that, in the $d$-dimensional integer lattice \eqref{lattice} with a total of $n = m^d$ nodes, there are at most $n (2d)^{k-1}$ paths of length $k$.
Following Remark~\ref{rem:negative-psi}, we can handle the case where $|\psi|$ is large enough in a similar way.
We thus arrive at the following.

\begin{cor} \label{cor:unknown}
Consider the context of \corref{known} but now assuming that the starting location is unknown and that $k/\log m \ge C_0$ for some constant $C_0 >0$.
There exist constants $0 < C_1\le C_2 < 1$ depending only on $d$ and $C_0$ such that, when $|\psi| < C_1$, $\lim_{m \to \infty} R^*_{\cC,\psi} = 1$, while when $|\psi| > C_2$, $\lim_{m \to \infty} R^*_{\cC,\psi} = 0$.
\end{cor}

\begin{rem}
As in Remark~\ref{rem:conj1}, we conjecture that we may take $C_1 = C_2$, defined as in \eqref{C-ddag}.  (This definition would lead to a different constant in the present setting.) 
\end{rem}

\begin{rem}
The constants in \corref{known} and \corref{unknown} are implicit.  Our analysis provides some nontrivial bounds on these constants.  However, it is not precise enough to lead to the exact values.  
We note that the same is true in the (simpler) detection-of-means setting \cite{maze}.  
Also, we reveal a regime where the minimal correlation coefficient $\psi$ allowing asymptotic hypotheses separation (i.e., $R^*_{\cC,\psi}\to 0$) is bounded away from $0$ and $1$ when both $k$ and $n$ go to infinity.
(The situation is qualitatively different in dimension $d=2$ and we refer the reader to \cite{maze} for a detailed treatment of that case in the detection-of-means setting.)
\end{rem}

%\section{Discussion} \label{sec:discussion}
%
%We leave open some questions and generalizations of interest.
%
%{\em More refined results.}
%We leave behind the delicate and interesting problem of finding the exact detection rates with tight multiplicative constants.  Such results are available in the detection-of-means setting.  
%
%This is particularly appealing for simple settings such as finding an interval of autoregressive process, as described in \secref{intro-tseries}.
%Our proof techniques are not sufficiently refined to get such sharp bounds.  
%We already know that, in the detection-of-means setting, bounding the variance of the likelihood ratio does not yield the right constant.
%The variant which consists of bounding the first two moments of a carefully truncated 
%likelihood ratio, used in \cite{Ingster99}, is applicable here, but the calculations are quite complicated and we leave them for future research.
%
%{\em Texture over texture.}
%Throughout the paper we assumed that the background is Gaussian white noise. 
%This is not essential, but makes the narrative and results more accessible.
%A more general, and also more realistic setting, would be that of detecting a region where the dependency structure is markedly different from the remainder of the image.  This setting has been studied in the context of time series, for example, in some of the references given in \secref{intro-tseries}.  However, we are not aware of existing theoretical results in higher-dimensional settings such as in images.

\section{Proofs} \label{sec:proofs}

\subsection{Preliminaries}

Let $\bGamma(\psi)$ denote the covariance operator of an (infinite) autoregressive model of order 1 with parameter coefficient $\psi$.
The operator $\bGamma(\psi)$ is positive definite and invertible when $|\psi|<1$.  Note that any $S \in \cC$ is homomorphic to $\{1, \dots, |S|\}$, and identifying the two, we have $(\bGamma_S(\psi))_{i,j} = \psi^{|i-j|}$, where $\bGamma_S(\psi)$ is the principal submatrix of the covariance operator $\bGamma(\psi)$
indexed by $S$.

Any autoregressive process $Y=(Y_i)_{i\in \bbZ}$  of order $1$ with parameter $\psi$  can be represented as a Gaussian Markov random field (GMRF) on the line. We have the decomposition 
\beq\label{eq:definition_GMRF}
Y_i = \phi Y_{i-1}+ \phi Y_{i+1}+ \epsilon_i \ , 
\eeq
where $\epsilon_i\sim \cN(0,\sigma_{\phi}^2)$ is independent of $(Y_j)_{j\neq i}$ and 
\beq\label{eqdef:phi}
\phi := \frac{\psi}{1+\psi^2}\ , \quad \quad \quad \sigma^2_{\phi} := \frac{1-\psi^2}{1+\psi^2}\ . 
\eeq
We start with some simple remarks relating autoregressive processes of order $1$ to GMRFs. See
\cite[Sect. 1.3]{MR1344683} for more details. This representation of a stationary autoregressive process  enables us to adapt some of the arguments developed in  \citep{arias2015detecting} for stationary GMRFs.
%The parameter vector is $\phi = (\phi_{-1}, 0, \phi) \in \mathbb{R}^{\bbN_1}$ with $\phi_{-1}=\phi$, and throughout this section, $\phi$ will have that meaning.
 
\begin{lem}
\label{lem:control_Gamma_S}
Identify $S$ with $(1,\ldots , |S|)$ and consider any $\psi\in (-1,1)$.  Then
\[
(\bGamma_S^{-1}(\psi))_{i,j} =
\left\{ \begin{array}{ll} 
1/\sigma^2_{\phi} & \text{if } i=j \text{ and  }i\in \{2,\ldots, |S|-1\}, \\
1/(1-\psi^2) & \text{if } i\in \{1,|S|\}, \\
- \phi/\sigma_{\phi}^2 & \text{if }|i-j|=1, \\
0 & \text{if }|i-j|>1.
\end{array}\right.
\]
\end{lem}

\begin{proof}
We leave $\psi$ implicit and write $\bGamma$ for $\bGamma(\psi)$.
Denote by $Y_S$ the restriction of the stationary autoregressive process $Y$ to $S$. First consider any index $i\in \{2,\ldots, |S|-1\}$. 
By the Markov property \eqref{eq:definition_GMRF}, conditionally to $(Y_{i-1},Y_{i+1})$, $Y_i$ is
independent to all the remaining variables. Thus, the conditional
distribution of $Y_{i}$ given $(Y_j)_{j\neq i}$ is the same as the conditional
distribution of $Y_{i}$ given $(Y_{j})_{j\in S\setminus\{i\}}$. This
conditional distribution characterizes the $i$-th row of the inverse covariance matrix
$\bGamma^{-1}_S$. More precisely,  the conditional variance $\sigma^2_{\phi}$ of $Y_i$ given $Y_{S}$ is $[(\bGamma_S^{-1})_{i,i}]^{-1}$. Furthermore, $-(\bGamma^{-1})_{i,j}/(\bGamma^{-1})_{i,i}$ is the $j$-th parameter of the conditional regression \eqref{eq:definition_GMRF} of $Y_i$ given $(Y_j)_{j\neq i}$,
and therefore we conclude that
 $(\bGamma^{-1})_{i,i}= (\sigma^2_{\phi})^{-1}=(\bGamma_S^{-1})_{i,i} $ and $(\bGamma^{-1})_{i,j}/(\bGamma^{-1})_{i,i}$ equals $-\phi$ if $|j-i|=1$ and is zero otherwise.
 
Now consider the case $i=|S|$, $i=1$ being handled similarly. Since $\bGamma^{-1}_{S}$ is a symmetric matrix, we only have to compute $(\bGamma^{-1}_{S})_{|S|,|S|}$ and $(\bGamma^{-1}_{S})_{|S|,1}$. By definition of  autoregressive processes, we have
\[Y_{|S|}= \psi Y_{|S|-1}+ \omega_{|S|}\ , \]
where $\omega_{|S|}\sim \cN(0, 1- \psi^2)$ is independent of $(Y_1,\ldots Y_{|S|-2})$. The above expression characterizes the conditional regression of $Y_{|S|}$ given $(Y_1, \ldots Y_{|S|-1})$. Arguing as previously, we conclude that  
 $(\bGamma^{-1})_{|S|,1}= 0$ and $(\bGamma^{-1})_{|S|,|S|}= 1/(1-\psi^2)$.
 \end{proof}

We let $\|\bA\|$ denote the operator norm of a matrix $\bA$.  
\begin{lem} \label{lem:trace}
Let $\bA$ and $\bB$ be (complex or real) matrices of same dimensions.  Let ${\rm col}(\bA)$ index the column vectors of $\bA$ that are nonzero.  Then
\[
|\tr(\bA^\top \bB)| \le |{\rm col}(\bA) \cap {\rm col}(\bB)| \|\bA\| \|\bB\|~.
\] 
\end{lem}

\begin{proof}
Define the index set $J = {\rm col}(\bA) \cap {\rm col}(\bB)$.  
Let $\bA_{J}$ denote the submatrix of $\bA$ with columns indexed by $J$, and define $\bB_J$ similarly.  
We then have 
\[|\tr(\bA^\top \bB)| = |\tr(\bA_J^\top \bB_J)| \le |J| \|\bA_J^\top \bB_J\| \le |J| \|\bA_J\| \|\bB_J\| \le |J| \|\bA\| \|\bB\|~. \qedhere\] 
\end{proof}

\subsection{Proof of \thmref{graph}}

Recall that with $\psi$ fixed, $\bGamma_S$ denotes the covariance matrix of an autoregressive model of order $1$ of length $|S|$ and with parameter $\psi$.
As is well-known (and explained in \citep{arias2015detecting} for example), 
\beq \label{LR-risk}
\bar{R}^*_\nu = 1 - \frac12 \E_0 | L_\nu(\X) - 1 | \geq 1 - \frac12 \sqrt{\E_0[L_\nu(\X)^2] -1} ~,
\eeq
where
\[
L_\nu(x) := \sum_{S\in \cC} \nu(S) L_S(x)~,
\]
with
\beq \label{L_S}
L_S(x) := \exp\left(\tfrac12 x_S^\top (\bI_S - \bGamma_{S}^{-1}) x_S - \tfrac12 \log \det(\bGamma_{S})\right)~.
\eeq
We thus need to upper bound
\beq\label{L2}
\E_0[L_\nu(\X)^2] = \sum_{S,T\in \cC} \nu(S) \nu(T) \E_0[L_S(\X) L_T(\X)]~.
\eeq
Unlike in the setting that concerns \citep{arias2015detecting}, here two subgraphs $S$ and $T$ in the support of the prior $\nu$ may not be disjoint, and if $S \cap T \ne \emptyset$, $L_S(\X)$ and $L_T(\X)$ are not independent.  

Before we proceed, we leave the dependency in $\X$ implicit and we formally index the elements of $S$ by $\{1, \dots, k\}$. 
By \lemref{control_Gamma_S}, we have 
\[1\leq \left(\bGamma_S^{-1}\right)_{1,1} = \left(\bGamma_S^{-1}\right)_{k,k}= \frac{1}{1-\psi^2} \leq \left(\bGamma_S^{-1}\right)_{i,i}=\frac{1}{\sigma_{\phi}^2}~,\quad \forall i\in \{2,\ldots k-1\}~,\]
and
\[\left(\bGamma_S^{-1}\right)_{i,i+1}= \left(\bGamma_S^{-1}\right)_{i,i-1} = -\frac{\phi}{\sigma^2_{\phi}}~,\]
while all the other entries of $\bGamma_S^{-1}$ are zero. 
Hence,   the matrix   $\bA_{S}:= \bI_S - \bGamma_{S}^{-1}$ satisfies
\beq\label{AS-norm}
\|\bA_{S}\|\leq \sup_{j\in S}\sum_{i\in S}|(\bA_S)_{i,j}| \leq \frac{2|\phi|}{\sigma^2_{\phi}}+ \frac{1-\sigma_{\phi}^2}{\sigma_{\phi}^2} = \frac{2|\psi|}{1-|\psi|}~,
\eeq
where we used the expressions \eqref{eqdef:phi} of $\phi$ and $\sigma_{\phi}$ . 
Let us fix $S$ and $T$ two anomalous subsets in $\cC$. In the following $\wt{\bGamma}_S$ (resp. $\wt{\bGamma}_{T}$) denotes the covariance of $X_{S\cup T}$ when $X\sim \P_{S,\psi}$ (resp. $X\sim \P_{T,\psi}$). Note that the restriction of $\wt{\bGamma}_S$ to $S\times S$ is exactly $\bGamma_S$ whereas its restriction to $(T\setminus S) \times (T\setminus S) $ is the identity matrix. 
We have
\begin{align*}
\E [L_S L_T]
&=  \E \left[\exp\left(X_{S\cup T}^\top (\bI_{S\cup T} - \tfrac12 \wt{\bGamma}_S^{-1} - \tfrac12 \wt{\bGamma}^{-1}_{T} ) X_{S\cup T} - \tfrac12 \log \det(\bGamma_S) - \tfrac12 \log \det(\bGamma_{T}) \right) \right]\\
&=  \exp\left(- \tfrac12 \log \det (\wt{\bGamma}^{-1}_{S} + \wt{\bGamma}^{-1}_{T} - \bI_{S\cup T}) - \tfrac12 \log \det(\bGamma_{S}) - \tfrac12 \log \det(\bGamma_{T}) \right)~.
\end{align*}
We used the fact that
\beq \label{LS_LT_1}
\|\bI_{S\cup T} - \tfrac12 \wt{\bGamma}_S^{-1} - \tfrac12 \wt{\bGamma}^{-1}_{T}\| \le \frac12 \|\bA_{S}\| + \frac12 \|\bA_{T}\| \le \frac{2|\psi|}{1-|\psi|} < \frac12~,
\eeq
by \eqref{AS-norm} and the fact that $|\psi| \le 1/5$. 
Define $\wt{\bA}_S= \bI_{S\cup T}- \wt{\bGamma}_S^{-1}$ and $\wt{\bA}_T$ similarly.
Using these bounds, together with the fact that, for a symmetric matrix $\bB$ with operator norm strictly less than 1, 
\[\log \det (\bI - \bB) = \tr \log (\bI - \bB) = - \sum_{\ell=1}^\infty \frac1\ell \tr(\bB^\ell)~,\] 
we get 
\begin{align*}
\Lambda :
&= - \tfrac12 \log \det (\wt{\bGamma}^{-1}_{S} + \wt{\bGamma}^{-1}_{T} - \bI_{S\cup T}) - \tfrac12 \log \det(\bGamma_{S}) - \tfrac12 \log \det(\bGamma_{T}) \\
&= - \tfrac12 \log \det (\bI_{S \cup T} - \wt{\bA}_{S} - \wt{\bA}_T) + \tfrac12 \log \det(\bI_{S \cup T} - \wt{\bA}_{S}) + \tfrac12 \log \det(\bI_{S \cup T} - \wt{\bA}_T)t \\
&= \frac12 \sum_{\ell=1}^\infty \frac1\ell \left(\tr\big[(\wt{\bA}_{S} + \wt{\bA}_{T})^\ell\big] - \tr\big[\wt{\bA}_{S}^\ell\big] - \tr\big[\wt{\bA}_{T}^\ell\big]\right) \\
&= \frac12 \sum_{\ell=2}^\infty \frac1\ell \sum_{(s,t) \in Q_\ell} \tr\big[\wt{\bA}_{S}^{s_1}\wt{\bA}_{T}^{t_1} \cdots \wt{\bA}_{S}^{s_\ell} \wt{\bA}_{T}^{t_\ell}\big]~,
\end{align*}
where 
\[Q_\ell := \Big\{(s,t) \in (\{0,1\}^\ell \setminus \{0\}^\ell)^2 :  
s_1 + \cdots + s_\ell + t_1 + \cdots + t_\ell = \ell\Big\}~.\]

For any $(s,t) \in Q_\ell$, there exists $j$ such that either $s_j = t_j = 1$ or $t_{j-1} = s_j = 1$.  For example, assuming the former holds, we apply \lemref{trace} to get
\begin{align*}
\tr\big[\wt{\bA}_{S}^{s_1} \wt{\bA}_{T}^{t_1} \cdots \wt{\bA}_{S}^{s_\ell} \wt{\bA}_{T}^{t_\ell}\big] 
&= \tr\big[\big(\wt{\bA}_{S}^{s_1} \wt{\bA}_{T}^{t_1} \cdots \wt{\bA}_{S}^{s_j}\big) \big(\wt{\bA}_{T}^{t_j} \cdots \wt{\bA}_{S}^{s_\ell} \wt{\bA}_{T}^{t_\ell}\big)\big] \\
&\le |S \cap T| \|\wt{\bA}_{S}^{s_1} \wt{\bA}_{T}^{t_1} \cdots \wt{\bA}_{S}^{s_j}\| \|\wt{\bA}_{T}^{t_j} \cdots \wt{\bA}_{S}^{s_\ell} \wt{\bA}_{T}^{t_\ell}\| \\
&\le |S \cap T| \|\wt{\bA}_{S}\|^{s_1} \|\wt{\bA}_{T}\|^{t_1} \cdots \|\wt{\bA}_{S}\|^{s_\ell} \|\wt{\bA}_{T}\|^{t_\ell} \\
&\le |S \cap T| \zeta^\ell~, \quad \text{where } \zeta = \frac{2|\psi|}{1-|\psi|}~.
\end{align*}
Note that the last line comes from \eqref{LS_LT_1}.

With this, together with the fact that $|Q_\ell| \le \binom{2\ell}{\ell}$ and $(1-x)^{-1/2} = \sum_{n \ge 0} \binom{2n}{n} (x/4)^n$ for $x\in (0,1)$, and noting that $\zeta < 1/4$ when $|\psi|<1/9$, we obtain
\[
\Lambda 
\le \frac12 |S \cap T| \sum_{\ell \ge 2} \frac1\ell \binom{2\ell}{\ell} \zeta^\ell \\
\le \frac14 |S \cap T| \big[ (1 - 4 \zeta)^{-1/2} - 1 - 2 \zeta\big]\ . 
\]
%\nv{Could you check the computations? I did not find the same constants. }
%For $0\leq x\leq 1/4$, we have $\sqrt{1-x}\geq 1- x/2 - x^2/(3\sqrt{3})$. 
%As $|\psi|\leq 1/33 $, $4\zeta\leq 1/4$ and 
%\begin{align*}
%\Lambda & \le& |S \cap T| \frac{2 \zeta^2}{\sqrt{1 - 4 \zeta}}\le 10 |S \cap T| \psi^2 ~.
%\end{align*}
We then conclude with \eqref{L2} and the expression of $\zeta$ in terms of $\psi$.

\subsection{Proof of \prpref{V}}

All through, we leave the dependence of $t$ in $k$ implicit. 

{\em Under the null.}
We first control $V_t^*$ under the null hypothesis.  For simplicity, assume that $k$ is even and decompose the statistics $V_{t,S}$ into $V_{t,S}= V_{1,t,S} + V_{2,t,S}$, where 
\[V_{1,t,S} := \sum_{j = 1}^{k/2} V_{t,S}(2j), \quad\text{ and }\quad V_{2,t,S} := \sum_{j = 2}^{k/2} V_{t,S}(2j-1)~,\] 
so that all the terms in $V_{1,t,S}$ (resp.~$V_{2,t,S}$) are independent. 
Define $V_{1,t}^{*} = \max_{S \in \cC} V_{1,t,S}$ and $V_{2,t}^{*} = \max_{S \in \cC} V_{2,t,S}$.  By symmetry, it suffices to bound $V_{1,t}^{*}$.

For any $S \in \cC$, we have $V_{1,t,S} \sim \Bin(k/2, p_t)$.  Thus, for any $S \in \cC$, with the union bound, we have 
\[
\P_0\{V_{1,t}^{*} \ge v\} \le |\cC| \P_0\{V_{1,t,S} \ge v\} \le |\cC| \P\big\{\Bin(k/2, p_t) \ge v\big\}~. 
\]
Define $b_t = 1/(2p_t)$ and $\varphi(b) = b(\log b -1) + 1$.
Choosing $v = k/4 = b_t k/2 p_t$, and using Bennett's inequality, the right-hand side is bounded by
\begin{align*}
|\cC| \exp\big(-k/2 p_t \varphi(b_t)\big) 
&=  \exp\big(\log |\cC| -(k/2) p_t \varphi(b_t)\big) \\
&= \exp\big(\log |\cC| -(k/2) \tfrac{1}{2}h(2p_t)\big) \\
&\le \exp\big((k/8 -(k/2) \tfrac{1}{2}) h(2p_t)\big) \\ 
%&\le  \exp\big(- (k/8) h(2p_t) \big) 
&\le  \exp(- k/8),
\end{align*}
where the inequality holds eventually as the sample size increases.  

Thus we found that
\begin{align*}
\P_0\{V_t^{*} \ge k/2\} 
&\le \P_0\{V_{1,t}^{*} \ge k/4\} + \P_0\{V_{2,t}^{*} \ge k/4\} \\
&\le \exp(- k/8) + \exp(- k/8) = 2 \exp(- k/8)~.
\end{align*}

\def\q{q}
\medskip\noindent {\em Under the alternative.}
We now consider the alternative hypothesis.  
Let $(S, \psi) \in \cC \times (-1,1)$ denote the anomalous path and the autoregressive parameter. 
Recall that we assume that $\psi\geq 1 - (t/\mathsf{N}^{-1}(4/5))^2$.
Denote $S = (s_1, \dots, s_k)$.
By definition, $V_t^* \ge V_{t,S}$.
Define   $Z_j:= (X_{s_{j+1}}-X_{s_j})/\sqrt{2(1-\psi)}$ for any $j\in\{1,\ldots, k-1\}$. 
We have $Z_j \sim \cN(0,1)$ and $\E[Z_j Z_{j'}] = \psi^{|j-j'|-1}(\psi-1)/2$ for $j\neq j'$. 
Define $\q =  2\mathsf{N}(t/\sqrt{1-\psi}) - 1$, and note that $q \ge 3/5$ by our assumption on $\psi$. 

First, in terms of expectation, we obtain the bound
\beq\label{eq:lower_E_SV}
\E_S\left[V_{t,S}\right]= (k-1)\q\geq (k-1) \frac35~.
\eeq

We now bound the variance.
Fix any $j\neq j'$, and denoting $a=\E[Z_j Z_{j'}]$, define $U= (Z_j-aZ_{j'})/\sqrt{1-a^2}$. 
%The next argument shows that for any $x$ such that $|x| \le t/\sqrt{1-\psi}$, the probability that $|Z_j|\leq t(1-\psi)^{-1/2}$ conditionally to $Z_i=x$ is close to $\q$. Indeed,
Note that $U \sim \cN(0,1)$ is independent of $Z_{j'}$.
For any $x \in \bbR$, we have
\begin{align*}
\P\left\{|Z_j|\leq \frac{t}{\sqrt{1-\psi}}~\Big|~Z_{j'}=x \right\} &=  \P\left\{\frac{-t-ax}{\sqrt{1-\psi}\sqrt{1-a^2}} \leq U\leq  \frac{t-ax}{\sqrt{1-\psi}\sqrt{1-a^2}}\right\}
\\
&\le  \P\left\{| U|\leq  \frac{t}{\sqrt{1-\psi}\sqrt{1-a^2}}\right\} \\
%&=  2 F\left(\frac{t}{\sqrt{1-\psi}\sqrt{1-a^2}}\right) - 1 \\
%&=  \q+ 2\Big[F\left(\frac{t}{\sqrt{1-\psi}\sqrt{1-a^2}}\right) - F\Big[\frac{t}{\sqrt{1-\psi}}\Big] \Big] \\
&=  2\P\left\{U \leq \frac{t}{\sqrt{1-\psi}} \right\} -1 + 2\P\left\{ \frac{t}{\sqrt{1-\psi}} \leq  U\leq \frac{t}{\sqrt{1-\psi}\sqrt{1-a^2}} \right\}\\
&=  \q+ 2\left[\mathsf{N}\left(\frac{t}{\sqrt{1-\psi}\sqrt{1-a^2}}\right) - \mathsf{N}\Big(\frac{t}{\sqrt{1-\psi}}\Big) \right] \\
%&\le  \q + 2\frac{ta^2}{\sqrt{1-\psi}\sqrt{1-a^2}} F'\Big(\frac{t}{\sqrt{1-\psi}}\Big)\\
&\le  \q + 2 \frac{t}{\sqrt{1-\psi}} \Big[\frac1{\sqrt{1-a^2}} - 1\Big] \frac1{\sqrt{2 \pi}} \exp\left[-\frac{t^2}{2(1-\psi)}\right]\\
&\le  \q + \big(\tfrac{4}{3})^{3/2} a^2\frac{t}{\sqrt{1-\psi}}  \frac1{\sqrt{2 \pi}} \exp\left[-\frac{t^2}{2(1-\psi)}\right]\\
&\le  \q + a^2~.
\end{align*}
We used the fact that $U$ is standard normal in the second and fourth line, a Taylor development (of order 1) in the fifth line, the fact that $|a|\leq 1/2$ in the sixth line, and the fact that $y e^{-y^2}\leq (2e)^{-1/2}$ for any $y \ge 0$ in the last line. 
As a consequence, 
\begin{align*}
\Cov\Big(\IND{|Z_j|\leq \tfrac{t}{\sqrt{1-\psi}}},\IND{|Z_{j'}|\leq \tfrac{t}{\sqrt{1-\psi}}}\Big)
&= \int_{- \frac{t}{\sqrt{1-\psi}}}^{\frac{t}{\sqrt{1-\psi}}} \P\left\{|Z_j|\leq \frac{t}{\sqrt{1-\psi}}~\Big|~Z_{j'}=x \right\}  \frac{e^{-x^2/2}}{\sqrt{2\pi}} {\rm d}x - \q^2 \\
&\le (\q +  a^2) \q - \q^2 
= \q a^2 = \q \psi^{2|j-j'|-2}(1-\psi)^2 ~,
\end{align*}
for all $j\ne j'$. 
We conclude that 
\begin{align*}
\Var_S\left(V_{t,S}\right)
&= \Var_S\left(\sum_{j=1}^{k-1}\IND{|Z_j|\leq \frac{t}{\sqrt{1-\psi}}}\right)\\
&\le (k-1)\left[\q(1-\q) + 2\q(1-\psi)^2  \sum_{i=1}^{\infty}\psi^{2(j-1)}\right]\\
&\le  (k-1)\left[\q(1-\q) + 2\q\frac{1-\psi}{1+\psi}\right] \\
&\leq 3\E_S\left[V_{t,S}\right]~.
\end{align*}

Now, by Chebyshev's inequality, we have for $b > 0$,
\[
\P_S\Big\{V_{t,S} \ge \E_S\left[V_{t,S}\right] - b \sqrt{\Var_S\left[V_{t,S}\right]}\Big\} \ge 1 - 1/b^2~,
\]
with 
\begin{align*}
\E_S\left[V_{t,S}\right] - b \sqrt{\Var_S\left[V_{t,S}\right]}
&\ge \E_S\left[V_{t,S}\right] - b \sqrt{3 \E_S\left[V_{t,S}\right]}~,
\end{align*}
and choosing $b = \log k$, for $k$ large enough we have, in view of \eqref{eq:lower_E_SV},
\[
\E_S\left[V_{t,S}\right] - b \sqrt{3 \E_S\left[V_{t,S}\right]} \ge k/2~.
\]
When this is the case, we have
\[
\P_{S, \psi}\{V_t^* < k/2\} 
\le \P_{S, \psi}\{V_{t,S} < k/2\}
\le 1/b^2 = 1/(\log k)^2~.
\]

{\em Risk.} 
Thus, for $k$ large enough, say $k > k_0$ (where $k_0$ is a numerical constant), the risk of $f_t$ is bounded by $2 e^{-k/8} + 1/(\log k)^2$.  We conclude that the stated result holds with $r_k = 2$ for $k \le k_0$ and $r_k =   2 e^{-k/8} + 1/(\log k)^2$ for $k > k_0$.

\subsection*{Acknowledgements}

EAC was partially supported by the US National Science Foundation (DMS 1223137, 1120888).
GL was partially supported by the Spanish Ministry of Economy and Competitiveness, Grant MTM2015-67304-P and FEDER, EU.
NV was partially supported by the French Agence Nationale de la Recherche (ANR 2011 BS01 010 01 projet Calibration).

\bibliographystyle{chicago}
\bibliography{../ref}

\end{document}